\documentclass[a4paper]{article}
\usepackage{amsthm,amssymb,amsmath,enumerate,graphicx,epsf}
\usepackage{authblk}

\newcommand{\COLORON}{1}
\newcommand{\NOTESON}{1}
\newcommand{\Debug}{0}
\usepackage[usenames]{color} 
\usepackage{amsthm,amssymb,amsmath,bbm,enumerate,graphicx,epsf,stmaryrd}
\usepackage[bookmarks, colorlinks=false, breaklinks=true]{hyperref} 

\newcommand{\comment}[1]{}
\newcommand{\COMMENT}[1]{}

\definecolor{darkgray}{rgb}{0.3,0.3,0.3}
\newcommand{\defi}[1]{{\color{darkgray}\emph{#1}}}

\comment{
	\begin{lemma}\label{}	
\end{lemma}
\begin{proof}

\end{proof}

\begin{theorem}\label{}
\end{theorem} 
\begin{proof} 	

\end{proof}

}

\newtheorem{proposition}{Proposition}[section]
\newtheorem{definition}[proposition]{Definition}
\newtheorem{theorem}[proposition]{Theorem}
\newtheorem{corollary}[proposition]{Corollary}

\newtheorem{lemma}[proposition]{Lemma}

\newtheorem{conjecture}{{Conjecture}}[section]

\newtheorem{problem}[conjecture]{{Problem}}

\newtheorem{examp}[proposition]{Example}

\newcommand{\FIG}{0}

\ifnum \NOTESON = 1 \newcommand{\note}[1]{ 

\hspace*{-30pt}
	{\color{blue}  NOTE: \color{Turquoise}{\small  \tt \begin{minipage}[c]{1.1\textwidth}  #1 \end{minipage} \ignorespacesafterend }} 
	
	}
\else \newcommand{\note}[1]{} \fi

\newcommand{\afsubm}[1]{ \ifnum \Debug = 1 {\mymargin{#1}}
\fi} 

\ifnum \Debug = 1 
\else  \fi

\ifnum \FIG = 1 
\else  \fi

\ifnum \FIG = 1 
\else  \fi

\ifnum \Debug = 1 \usepackage[notref,notcite]{showkeys}
\fi

\ifnum \COLORON = 0 \renewcommand{\color}[1]{}
\fi

\newcommand{\R}{\ensuremath{\mathbb R}}
\newcommand{\Z}{\ensuremath{\mathbb Z}}

\newcommand{\cp}{\ensuremath{\mathcal P}}

\newcommand{\eps}{\ensuremath{\epsilon}}

\newcommand{\lam}{\ensuremath{\lambda}}

\newcommand{\ccc}{\ensuremath{\mathcal C}}

\newcommand{\sm}{\backslash}

\newcommand{\sydi}{\triangle}

\newcommand{\g}{\ensuremath{G\ }}
\newcommand{\G}{\ensuremath{G}}

\newcommand{\Lr}[1]{Lemma~\ref{#1}}

\newcommand{\Tr}[1]{Theorem~\ref{#1}}

\newcommand{\Sr}[1]{Section~\ref{#1}}

\newcommand{\Prr}[1]{Pro\-position~\ref{#1}}

\newcommand{\Cr}[1]{Corollary~\ref{#1}}

\renewcommand{\iff}{if and only if}
\newcommand{\fe}{for every}

\newcommand{\st}{such that}

\newcommand{\ti}{there is}

\newcommand{\mymargin}[1]{
  \marginpar{%
    \begin{minipage}{\marginparwidth}\small%
      \begin{flushleft}%
        {\color{blue}#1}%
      \end{flushleft}%
   \end{minipage}%
  }%
}%

\newcommand{\mySection}[2]{}

\newcommand{\dir}{\overrightarrow}

\newcommand{\meg}{\geqslant}
\newcommand{\mik}{\leqslant}
\newcommand{\ltest}{$L_1$--testable}

\title{The Bradley--Terry condition is $L_1$--testable}

\author[1]{Agelos Georgakopoulos\thanks{Supported by EPSRC grant EP/L002787/1, and by the European Research Council (ERC) under the European Union’s Horizon 2020 research and innovation programme (grant agreement No 639046).}}
\author[2]{Konstantinos Tyros\thanks{Supported by ERC
grant 306493.}}
\affil[1]{{Mathematics Institute}\\
	{University of Warwick}\\
	{CV4 7AL, UK}\\}
\affil[2]{Department of Mathematics, Ko\c{c} University, Sarıyer, Istanbul, 34450, Turkey}

\date{}
\begin{document}
\maketitle

\begin{abstract}
We provide an algorithm with constant running time that given a weighted tournament $T$, distinguishes with high probability of success between the cases that $T$ can be represented by a Bradley--Terry model, or cannot even be approximated by one. The same algorithm tests whether the corresponding Markov chain is reversible.
\end{abstract}


\section{Introduction}

Suppose we have a set $S$ of individuals participating in a tournament, and for each pair $x,y$ in $S$ we have assigned a probability $p_{xy}$ that $x$ beats $y$ (so that $p_{yx} = 1- p_{xy}$). The Bradley--Terry model \cite{BraTer} seeks to assign real numbers $a(x)>0$ to these individuals so that
\begin{equation}
\label{pa}
\frac{p_{xy}}{p_{yx}} = \frac{a(x)}{a(y)}.
\end{equation}
holds for every $x,y$ in $S$; equivalently, we have $p_{xy} = \frac{a(x)}{a(x)+a(y)}$.

This statistical model, which was already considered by Zermelo \cite{ZerBer}, is used in many practical applications where individuals can be compared in pairs, and the probabilities $p_{xy}$ are often measured empirically\footnote{{\small en.wikipedia.org/wiki/Bradley-Terry\_model}}.

Although it is convenient to represent the ${|S|\choose 2}$ parameters $p_{xy}$ with the just $|S|$ parameters $a(x)$, it is easy to see that not all tournaments can be approximated well in this way: consider for example a tournament containing three vertices $x_0,x_1,x_2$ such that $x_i$ beats $x_{i+1\pmod 3}$ with probability 90\%. The model has a widespread use nevertheless, and the algorithm used for computing the $a(x)$ is known to converge even for such tournaments \cite{DavMeth}.

Thus it is important to know when a Bradley--Terry model approximates the probabilities $p_{xy}$ well-enough. An obvious but rather inefficient approach is to first compute the $a(x)$, and then test them against \eqref{pa} for every pair $x,y$.

\medskip
The aim of this paper is to provide a fast randomized algorithm that given a tournament $T$, distinguishes with high probability of success between the cases that $T$ can be represented by a Bradley--Terry model, or cannot even be approximated by one. The running time of our algorithm is independent of the size of $T$; it only depends on the approximation tolerance and the desired probability of success. It is assumed here that ---unlike in statistical setups--- our algorith can query the exact value $p_{xy}$ in one step.

More consisely, we prove that the Bradley--Terry condition is \ltest\ in the sense of \cite{BRY}. 

\section{Preliminaries}

\subsection{Property testing}
The notion of property testing  was established by Goldreich, Goldwasser and Ron \cite{GoGoRoPro} for graph properties, following a similar notion of Rubinfeld  and Sudan \cite{RuSuRob} for program     testing in computer science. Since then the concept has received a lot of attention in various contexts; see \cite{Goldreich}. We recall the following standard definitions from \cite{AlFoEas}.

A \defi{property} $\cp$ is a family of (undirected)
graphs closed under isomorphism. A graph \G\ with n vertices is \defi{$\epsilon$-far} from satisfying $\cp$ if one must
add or delete at least $\epsilon n^2$ edges in order to turn \G\ into a graph in $\cp$.

An \defi{$\epsilon$-tester} for $\cp$ is a randomized algorithm, which given the ability to check whether there is
an edge between a given pair of vertices, distinguishes with probability at least 2/3 between the cases
$G$ satisfies $\cp$ and $G$ is $\epsilon$-far from satisfying $\cp$. Such an $\epsilon$-tester is \defi{one-sided} if, whenever $G$ satisfies
$\cp$, the $\epsilon$-tester determines this with probability 1. A property $\cp$ is \defi{strongly-testable} if for every fixed $\epsilon> 0$ there exists a one-sided $\epsilon$-tester for $\cp$ whose query complexity is bounded only by a function of
$\epsilon$, which is independent of the size of the input graph.
Call a property $\cp$ \defi{easily testable} if it is strongly testable with a one-sided $\epsilon$-tester whose query complexity is polynomial in $\epsilon^{-1}$, and otherwise call $\cp$ \defi{hard}.

It is easy to adapt the above definitions to the class of edge-weighted graphs: say that a weighted graph $(G,w)$ with n vertices is \defi{$\epsilon$-far} from satisfying $\cp$ if the total change in $w$ required
 in order to turn $(G,w)$ into a weighted graph in $\cp$ is  at least $\epsilon n^2$. All other definitions can be applied verbatim and, following \cite{BRY}, we say that a property is (easily/strongly) \ltest\ if it complies with the above definitions when we use this new notion of $\epsilon$-far.

\subsection{Stochastic tournaments} \label{secStTo}

A \defi{stochastic tournament} is a pair $(K,w)$ where $K$ is a \defi{tournament}, i.e.\ a complete graph with $n$ vertices in which every edge is given a direction, and $w: E(K) \to [0,1]$ is an assignment of weights to the (directed) edges of $K$. Intuitively, we think of $w(xy)$ as the probability that $x$ beats $y$ in a game between them. The \defi{matrix} $P= P(K,w) = (p_{xy})$ of the stochastic tournament  $(K,w)$ is the square matrix with rows and columns indexed by the vertices of $K$ \st\ $p_{xy} = w(xy) = 1-p_{yx}$ whenever the edge $xy$, directed from $x$ to $y$, belongs to $E(K)$. A stochastic tournament $(K,w)$  satisfies the \defi{Bradley--Terry condition}, if there exists a family of positive real numbers $(a_x)$ indexed by the vertex set of $K$ such that
condition \eqref{pa} holds.
\comment{
	for every pair of distinct nodes $x$ and $y$ we have that
\begin{equation}
  \label{pa}
  \frac{p_{xy}}{p_{yx}} = \frac{a(x)}{a(y)}.
	\end{equation}
}
Moreover, a stochastic tournament corresponds to a Markov chain by considering the matrix $Q=Q(K,w)=(q_{xy})_{x,y\in V(K)}$ defined by
\begin{equation}
q_{xy}=\left\{ \begin{array} {l} p_{xy}/n,\;\;\;\;\;\;\;\;\;\;\;\;\;\;\;\;\text{ if } x\neq y\\
                               1 - \frac1{n}\sum_{z\neq x} p_{xz}, \;\text{ if } x= y     \end{array}  \right.
\end{equation}
Then $Q$ defines indeed a Markov chain by interpreting $q_{xy}$ as the transition probability from vertex $x$ to vertex $y$.
Let us recall that such a Markov chain with transition matrix $(q_{xy})_{x,y\in V(K)}$ is called reversible, if there exists a probability distribution $(\pi_x)$ on the vertex set of $K$ such that for every pair of nodes $x,y$ we have
\begin{equation}
  \label{revdef}
  \pi_xq_{xy}=\pi_yq_{yx}.
\end{equation}
All these notions are related to each other as it is described by the following partly folklore proposition.
\begin{proposition}
    \label{prop1}
	Let $(K,w)$ be a stochastic tournament. Then the following conditions are equivalent
	\begin{enumerate}
		\item \label{BT} $(K,w)$ corresponds
		to a Bradley--Terry model;
		\item \label{rev} The Markov chain corresponding to $Q$ is reversible;
		\item \label{tri} Every triangle of $(K,w)$ is balanced, i.e.\ for every triple $x,y,z$ in $V(K)$, we have $p_{xy} p_{yz} p_{zx} = p_{xz} p_{zy} p_{yx}$;
		\item \label{trir} There is a vertex $r$ in $V(K)$ such that every triangle containing $r$ is balanced.
	\end{enumerate}
	
\end{proposition}

The equivalence of \ref{BT} and \ref{rev} was noticed by Rajkumar \& Agarwal \cite[Section 4]{RajAgr}, while the equivalence of \ref{rev} and \ref{tri} is essentially Kolmogorov's criterion.

\begin{proof}
	\ref{BT} $=$ \ref{rev}: Let the stationary distribution $\pi(x)$ of a vertex equal the inverse $1/a(x)$ of its Bradley--Terry score. Then equivalence follows easily from the definition of reversibility: we have $\pi(x) Q(xy) = \pi(y) Q(yx)$ \iff\ $P(xy)/P(yx) = a(x)/a(y)$.
	
	\ref{trir} $\to$ \ref{BT}: Given such a vertex $r$, let $a(r)= 1$. For every other vertex $y$, let $a(y):= p_{yr}/p_{ry}$. We claim that these numbers satisfy the Bradley--Terry condition, i.e.\ $p_{zy}/p_{yz} = a(z)/a(y) $ \fe\ $y,z$ in $V(K)$.	
	Indeed, note that $\frac{p_{ry}}{p_{yr}}   \frac{p_{zr}}{p_{rz}}= \frac1{a(y)} \frac{a(z)}{1}$ by the definition of $a$. Combining this with \ref{trir} applied to the triangle $r,y,z$ results in cancellations yielding the desired $p_{zy}/p_{yz} = a(z)/a(y) $.

	\ref{BT} $\to$ \ref{tri}: We have $\frac{p_{xy} p_{yz} p_{zx} }{ p_{xz} p_{zy} p_{yx}} = \frac{a(x)a(y)a(z)}{a(x)a(z)a(y)}=1$ by applying  \eqref{pa} three times. Finally, the implication \ref{tri} $\to$ \ref{trir} is trivial completing the proof.
\end{proof}

\begin{theorem}
\label{test}
Each of the four conditions of \Prr{prop1} is easily \ltest.
\end{theorem}

As observerd in \cite[Fact 1.1.]{BRY}, $L_1$--testability implies $L_p$ testability for all $p>1$.
The proof of Theorem \ref{test} is based on the balanced triangle condition of \Prr{prop1}~\ref{tri}, and parallels a classical proof of testability of triangle-freeness for dense graphs. For this we use a continuous analogue of the triangle removal  lemma, which we prove in the next section (\Cr{bac}).

In \Sr{secEps} we show that the notion of being $\epsilon$-far from satisfying these conditions corresponds to natural approximate versions of the Bradley--Terry condition and reversibility.


\section{Proof of \Tr{test}}

Let us observe that given any stochastic tournament $(K,w)$ and any triangle $T$ with nodes $x,y$ and $z$, we can make $T$ balanced by changing the value $w$ on any of its edges. More precisely, $T$ can be made balanced by changing, for instance, the value of $w$ only on the edge $xy$ yielding $p_{xy}$ to be equal to $\frac{p_{zy}p_{xz}}{p_{zy}p_{xz}+p_{yz}p_{zx}}$.

\begin{definition}
  Let $(K,w)$ be a stochastic tournament and $T$ a triangle with nodes $x,y$ and $z$. We define the discrepancy of $T$, denoted by $\mathrm{disc}_{(K,w)}(T)$ to be the maximum of the changes of $w$ required on a single edge of $T$ to turn it into a balanced triangle, that is,
  \begin{equation}
    \mathrm{disc}_{(K,w)}(T)=    \max \big\{|\alpha|,|\beta|,|\gamma|\big\}
  \end{equation}
  where $\alpha=p_{xy}-\frac{p_{zy}p_{xz}}{p_{zy}p_{xz}+p_{yz}p_{zx}}$,
  $\beta=p_{yz}-\frac{p_{yx}p_{xz}}{p_{yx}p_{xz}+p_{xy}p_{zx}}$ and $\gamma=p_{zx}-\frac{p_{yx}p_{zy}}{p_{yx}p_{zy}+p_{xy}p_{yz}}$.
\end{definition}

We will omit the subscript from the symbol $\mathrm{disc}_{(K,w)}$ whenever the stochastic tournament $(K,w)$ is clear from the context.
Let us observe that $T$ is balanced if and only if $\mathrm{disc}(T)=0$.

The following lemma is based on the observation that we can make any stochastic tournament reversible by modifying the edges that do not belong to any fixed spanning tree; see \Lr{ST}. Here we work with the spanning tree consisting of all edges incident with a fixed vertex $r$.
\begin{lemma}
  \label{lem1}
  Let $(K,w)$ be a stochastic tournament and $r$ a node in $K$. Also let $\mathcal{C}$ be the collection of all triangles containing $r$ and for each $T$ in $\mathcal{C}$ let $e_T$ be the edge of $T$ that does not contain $r$. Then \ti\  a weighting $w'$ satisfying the following:
  \begin{enumerate}
    \item \label{it1} the stochastic tournament $(K,w')$ is reversible,
    \item \label{it2} $w'(e)=w(e)$ for all $e$ in $E(K)$ that do not belong to $\{e_T:T\in\mathcal{C}\}$ and
    \item \label{it3} $|w'(e_T)-w(e_T)|\mik\mathrm{disc}_{(K,w)}(T)$ for all $T$ in $\mathcal{C}$.
  \end{enumerate}
\end{lemma}
\begin{proof}
  By the definition of $\mathrm{disc}_{(K,w)}(T)$, for each $T$ in $\mathcal{C}$, we can assign to $e_T$ a new weight $w'(e_T)$ making $T$ balanced and satisfying $|w(e_T)-w'(e_T)|<\mathrm{disc}_{(K,w)}(T)$. Since for every $T$ in $\mathcal{C}$ we know that $T$ is the only triangle in $\mathcal{C}$ containing $e_T$, we deduce that every triangle in $\mathcal{C}$ is balanced. The result follows by Proposition \ref{prop1}.
\end{proof}


Averaging this over all vertices $r$ of $K$, we deduce

\begin{corollary} \label{bac}
	Let $(K,w)$ be a stochastic tournament satisfying
\[\sum_{T}\mathrm{disc}(T)\mik\eps{|K|\choose 3},\]
for some $\eps>0$,
where the sum is taken over all triangles. Then \ti\ a weighing $w'$ such that the tournament $(K,w')$ is reversible, and
  \begin{equation}
  	\label{ww}
  	\sum_{e\in E(K)}|w(e)-w'(e)|< \epsilon {|K|\choose 2}.
  \end{equation}
\end{corollary}
\begin{proof}
  Let $\mathcal{T}$ be the collection of all triangles in $K$ and for every node $r$ in $K$, let $\mathcal{C}_r$ be the collection of all triangles containing $r$. Let $n=|K|$. Observe that
  \begin{equation}
    \label{eq009}
    \sum_{r\in V(K)}\sum_{T\in\mathcal{C}_r}\mathrm{disc}(T)=3\sum_{T\in\mathcal{T}}\mathrm{disc}(T)<3\epsilon {n\choose 3} < n\eps{n\choose2}
  \end{equation}
  and therefore there is some $r$ in $V(K)$ such that $\sum_{T\in\mathcal{C}_r}\mathrm{disc}(T)<\eps{n\choose2}$. The result follows from Lemma \ref{lem1}.
\end{proof}

We now use the last result to prove our main theorem.

\begin{proof}
  [Proof of Theorem \ref{test}]
  We consider the following randomised algorithm. Sample $f(\eps)=\lceil-\log_{1-\epsilon}3\rceil$ triangles of $K$ independently and uniformly at random, and check whether they are balanced. Answer `yes' if all these triangles are balanced, and answer `no' otherwise.

  Clearly, if $(K,w)$ satisfies the conditions of Proposition \ref{prop1}, our algorithm
  responds `yes' with probability one by item~\ref{tri}. On the other hand, assuming that  $(K,w)$ is $\epsilon$-far from satisfying these conditions, that is, \eqref{ww} fails for every $w'$ such that $(K,w')$ satisfies the Bradley--Terry condition,
   Corollary \ref{bac} implies
  \[\sum_{T}\mathrm{disc}(T)>\eps{|K|\choose 3}.\]
  Since $0\mik \mathrm{disc}(T)\mik 1$, letting $B$ be the set of the unbalanced triangles, we obtain $|B|\meg\eps{|K|\choose 3}$. By the choice of $f(\eps)$, our algorithm responds `no' with probability at least 2/3; indeed, each of our sampled triangles is in $B$ with probability at least $\epsilon$. Thus the probability that none of them is in $B$ is at most $ (1-\epsilon)^{f(\eps)} \mik 1/3$.
  It is not hard to check that $f(\eps)< c\eps^{-1}$ for some constant $c$ (in fact we can take $c=2$), and so our property is easily \ltest.
\end{proof}

{\bf Remark 1:} The above proof parallels the classical proof of testability of triangle-freeness for dense graphs, with \Cr{bac} playing the role of the triangle removal lemma. But as the same $\eps$ appears in the condition and the conclusion of \Cr{bac}, the conditions of Proposition \ref{prop1} are easily testable contrary to triangle-freeness which is known to be hard \cite[p.~4]{ConFoxGra}.

{\bf Remark 2:} Our proof in fact shows the slightly stronger fact that the conditions of \Prr{prop1} are easily $L_0$--testable in the sense of \cite[Fact 1.1.]{BRY}, but we chose to work with $L_1$ as it is more natural in our setup. Indeed, we could have defined $\mathrm{disc}(T)$ to be 1 if $T$ is unbalanced and 0 otherwise, and the rest of our proof could be used verbatim.

\section{Approximate versions of \Prr{prop1}} \label{secEps}
Let $(K,w)$ be a stochastic tournament and  $\eps$ a real with $0<\eps\mik1$. We will say that $(K,w)$ is an $\eps$-approximate Bradley--Terry model if there is a map $a:V(K)\to \R_{>0}$ such that
\[(1+\eps)^{-1}\frac{a(x)}{a(x)+a(y)}\mik p_{xy}\mik (1+\eps)\frac{a(x)}{a(x)+a(y)}\]
for all pairs of distinct vertices $x,y$ in $V(K)$.
We will call the map $a$ an $\eps$-approximate Bradley--Terry score. Moreover, a triangle with nodes $x,y$ and $z$  will be called $\eps$-balanced if
\[(1+\eps)^{-1}\mik \frac{p_{xy}p_{yz}p_{zx}}{p_{yx}p_{zy}p_{xz}}\mik1+\eps.\]

\begin{proposition}
  \label{aprprop1}
  Let $(K,w)$ be a stochastic tournament and  $\eps$ a real with $0<\eps\mik1$. Also let $r$ be a node in $K$ and assume that every triangle containing $r$ is $\eps$-balanced. Then $(K,w)$ is an $\eps$-approximate Brandley--Terry model.
\end{proposition}
\begin{proof}
  We define $a:V(K)\to\R_{>0}$ setting $a(r)=1$ and for every node $y$ in $K$ with $y\neq r$ we set $a(y)=\frac{p_{yr}}{p_{ry}}$. We claim that $a$ is an
  $\eps$-approximate Bradley--Terry score. Indeed, it follows directly from the definition of $a$ that for every node $y$ in $K$ with $y\neq r$, we have  \[p_{ry}=\frac{a(r)}{a(r)+a(y)}\;\text{and}\;p_{yr}=\frac{a(y)}{a(r)+a(y)}.\]
  Let $x,y$ be two distinct nodes in $K$ different from $r$. Since the triangle consisting of the nodes $x,y$ and $r$ is $\eps$-balanced, we obtain
  \begin{equation}
    \label{eq010}
    (1+\eps)^{-1}\mik \frac{p_{yx}}{p_{xy}}\cdot \frac{p_{ry}}{p_{yr}} \cdot \frac{p_{xr}}{p_{rx}}= \frac{p_{yx}}{p_{xy}} \cdot a(x)\cdot \frac{1}{a(y)}\mik 1+\eps
  \end{equation}
  and therefore
  \begin{equation}
    \label{eq011}
    (1+\eps)^{-1}\frac{a(y)}{a(x)}\mik \frac{p_{yx}}{p_{xy}}\mik (1+\eps)\frac{a(y)}{a(x)}.
  \end{equation}
  By \eqref{eq011}, we get that
    \begin{equation}
    \label{eq012}
    \begin{split}
      (1+\eps)^{-1}\Big(1+\frac{a(y)}{a(x)}\Big)&
      < 1 + (1+\eps)^{-1}\frac{a(y)}{a(x)}
    \mik 1 + \frac{p_{yx}}{p_{xy}}\\
    & <  1 + (1+\eps)\frac{a(y)}{a(x)}
    \mik (1+\eps)\Big(1+\frac{a(y)}{a(x)}\Big)
    \end{split}
  \end{equation}
  and therefore
  \begin{equation}
    \label{eq013}
    (1+\eps)^{-1} \frac{1}{1+\frac{a(y)}{a(x)}}<\frac{1}{1 + \frac{p_{yx}}{p_{xy}}}<(1+\eps)\frac{1}{1+\frac{a(y)}{a(x)}}.
  \end{equation}
  Finally, since
  \[\frac{a(x)}{a(x)+a(y)}=\frac{1}{1+\frac{a(y)}{a(x)}}\;\text{and}\;p_{xy}=\frac{1}{1 + \frac{p_{yx}}{p_{xy}}},\]
  by inequality \eqref{eq013}, the result follows.
\end{proof}

  Let $(K,w)$ be a stochastic tournament and $\eps$ a real with $0<\eps\mik1$. We will say that $(K,w)$ is $\eps$-approximately reversible, if there is a probability distribution $(\pi_x)_x$ such that for every pair of distinct nodes $x$ and $y$ in $K$ we have
  \[(1+\eps)^{-1}\mik\frac{\pi_x p_{xy}}{\pi_y p_{yx}}\mik1+\eps.\]

  \begin{proposition}
    \label{aprprop2}
    Let $(K,w)$ be a stochastic tournament and  $\eps$ a real with $0<\eps\mik1$. Assume that $(K,w)$ is an $\eps$-approximate Brandley--Terry model. Then
    $(K,w)$ is $3\eps$-approximately reversible.
  \end{proposition}
  \begin{proof}
    Let $a$ be an $\eps$-approximate Brandley--Terry score. Without loss of generality we may assume that $\sum_x1/a(x)=1$, where the sum is taken over all nodes of $K$. For every node $x$ of $K$ we set $\pi_x=1/a(x)$. We claim that $(\pi_x)_x$ witnesses that $(K,w)$ is $\eps$-approximately reversible. Indeed, let $x$ and $y$ be to distinct nodes of $K$. First observe that since $a$ is an $\eps$-approximate Brandley--Terry score, we have that
    \begin{equation}
      \label{eq014}
      (1+\eps)^{-1}\frac{a(x)}{a(x)+a(y)}\mik p_{xy} \mik (1+\eps)\frac{a(x)}{a(x)+a(y)}
    \end{equation}
    and
    \begin{equation}
      \label{eq015}
      (1+\eps)^{-1}\frac{a(y)}{a(x)+a(y)}\mik p_{yx} \mik (1+\eps)\frac{a(y)}{a(x)+a(y)}.
    \end{equation}
    Thus, we have that
    \begin{equation}
      \label{eq0016}
      \frac{\pi_xp_{xy}}{\pi_yp_{yx}}=\frac{a(x)^{-1}p_{xy}}{a(y)^{-1}p_{yx}}\mik\frac{1+\eps}{(1+\eps)^{-1}}\mik 1 + 3\eps.
    \end{equation}
    Similarly it follows that $\frac{\pi_xp_{xy}}{\pi_yp_{yx}}\meg(1+3\eps)^{-1}$.
  \end{proof}

  Finally, it is immediate that if a stochastic tournament $(K,w)$ is a $\eps$-approximately reversible then every triangle is $7\eps$-balanced.

\section{Balanced cycles}

In this section we show that if all cycles in any basis of the cycle space of our stochastic tournament $(K,w)$ are balanced, then every cycle of $K$ is balanced.

\begin{lemma} \label{ST}
	Let $K$ be a tournament and $S$ an (undirected) spanning tree of $K$. Then every weighting $\hat{w}:E(S)\to \R_{>0}$ can be extended to a weighting $w:E(K)\to \R_{>0}$ of $K$ such that the Markov chain defined by $Q(K,w)$ is reversible.
\end{lemma}	

As we observe below, \Lr{ST} can be thought of as a special case of a more general result \Lr{ccc}.

\begin{proof}[Proof of \Lr{ST}]
	We will define the extension $w$ of $\hat{w}$ by first defining the stationary measure $\pi$ of the Markov chain of $Q(K,w)$.
	
	For this, fix a vertex $r$ of $K$, and let $\pi(r)=1$ (any positive value would do). For each neighbour $y$ of $r$ in $S$, we let $\pi(y)= \pi(r)\hat{w}(ry)/(1-\hat{w}(ry))$, where we set $\hat{w}(ry) = 1- \hat{w}(yr)$ if the $ry$--edge is directed from $y$ to $r$.
	
	We proceed recursively to assign a value $\pi(x)$ to each neighbour $x$ of $y$ except $r$, by the same formula:  $\pi(x)= \pi(y)\hat{w}(yx)/(1-\hat{w}(yx))$. This defines $\pi$. Now for every chord $e=xy$ of $S$, we let $w(xy)$ be the unique solution to $\frac{w(xy)}{1-w(xy)} = \frac{\pi(y)}{\pi(x)}$, that is,  $w(xy)= c/(1+c)$ where $c= \frac{\pi(y)}{\pi(x)}$.
	It follows that the measure $\pi$ is stationary for $Q(K,w)$, since we have $\pi(x) q_{xy} = \pi(x) \frac{w(xy)}{n} = \pi(y) \frac{w(yx)}{n} = \pi(y) q_{yx}$ by the definitions.
	\comment{

		Let $cd \in E(K)\sm E(S)$ be a chord of $S$, and let $w(cd)$ be the unique value making the fundamental cycle $C_{cd}$ of $cd$ balanced. This uniquely determines the extension $w$ of $\hat{w}$.
		
		We claim that every cycle of $K$ is balanced. This will follow from the well-known fact that the fundamental cycles of any spanning tree generate the cycle space of the graph \cite{diestelBook05}, combined with the following claim
		
		{\bf Claim:} If $D,F$ are two balanced cycles of $K$, then any cycle $C$ contained in their symmetric difference $D \sydi F$ is balanced.
		
	}
\end{proof}

Call a cycle $C$ of $K$ \defi{balanced}, if $\lam(\dir{C}):= \Pi_{xy\in \dir{C}} \frac{p_{xy}}{p_{yx}} = 1$, where $\dir{C}$ denotes any of the two possible cyclic orientations of $C$.

\begin{lemma} \label{ccc}
	Let $(K,w)$ be a stochastic tournament. Let $B$ be a basis of the cycle space $\ccc_\Z$ of $K$. If every element of $B$ is balanced, then every element of $\ccc$ is balanced (and hence the Markov chain defined by $Q(K,w)$ is reversible).
\end{lemma}	
\begin{proof}
	It is straightforward to check that if $\dir{C},\dir{D}\in \ccc_\Z$ then $\lam(\dir{C}+\dir{D}) = \lam(\dir{C}) \lam(\dir{D})$ by the definition of $\lam$. Thus any element of $\ccc_\Z$ generated by a balanced set is balanced, and the result follows.
		By \Prr{prop1}, $Q(K,w)$ is reversible in this case.
\end{proof}

We could have deduced \Lr{ST} from the last result as follows. For every chord $e=xy$ of $T$, we can assign a value $w(xy)$ such that the fundamental cycle of $e$ with respect to $T$ becomes balanced. The result then follows from the well-known fact that the fundamental cycles of any spanning tree generate the cycle space \cite{diestelBook05}.

\section{Open Problems}

Suppose that instead of a tournament $K$ we have an arbitrary directed graph $G$, and an assignment of weights $w: E(G) \to (0,1)$ to the edges of $G$. Then the Markov chain corresponding to $Q$ is still well-defined if we set $p_{xy}=0$ whenever $x,y$ do not form an edge of $G$. Thus we can ask if the Markov chain is reversible. Likewise, we can generalise the Bradley--Terry condition by demanding \eqref{pa} only when $x,y$ form an edge. We expect that using a version of the Szemeredi regularity lemma it is possible to show that these properties are \ltest\ for arbitrary $G$ following the approach of \cite{AFNS} or \cite{AlShaMon}, which characterises the (unweighted) graph properties that are testable. However, this would give far worse bounds than those we obtained for tournaments, which naturally leads to the following question.

\begin{problem}
	Are reversibility and the Bradley--Terry condition easily \ltest\ for arbitrary directed graphs?
\end{problem}

We remark that the problem of characterising easily testable graph properties is open \cite{AloFoxEas}.

The notion of testability has been adapted to sparse graphs \cite{GolRonPro}; the definition is the same, except that we say that \g is \defi{$\epsilon$-far} from having a property $\cp$ if one must
add or delete at least $\epsilon n$ (rather than $\epsilon n^2$) edges in order to turn \G\ into a graph having $\cp$.

\begin{problem}
	Are reversibility and the Bradley--Terry condition \ltest\ in the sparse sense?
\end{problem}

\medskip

Next, we propose a generalisation of the Bradley--Terry condition. Although it also makes sense for arbitrary graphs, we will formulate it for tournaments for simplicity.

Let us say that a stochastic tournament $(K,w)$ has \defi{Bradley--Terry dimension} (or \defi{BT--dimension} for short) $d$, if there is a family $(a_i)_{1\leq i \leq d}$ of $d$ functions $a_i : V(K) \to \R_{>0}$ such that, \fe\ $xy\in V(K)$, we have \begin{equation}
	\label{pad}
p_{xy} = \frac1{d}\sum_{1\leq i \leq d} \frac{a_i(x)}{a_i(x)+a_i(y)},
\end{equation}
and this is not true for any family of $d-1$ functions. Here $p_{xy}$ is determined by $w$  as explained in \Sr{secStTo}.
Thus $(K,w)$ has BT--dimension $1$ \iff\ it satisfies the Bradley--Terry condition.

Intuitively, if $(K,w)$ has BT--dimension $d$, then we can represent the probabilities $p_{xy}$ via a set of $d$ games, such that each player $x\in V(K)$ has a strength $a_i(x)$ in game $i$, and when players $x,y$ compete, one of these $d$ games is chosen uniformly at random, and $x,y$ play a match of that game.

\begin{problem}
	Is the property of having BT--dimension $d$ \ltest\ for $d>1$?
\end{problem}

One can think of BT--dimension as a generalisation of reversibility for Markov chains due to \Prr{prop1}. It would be interesting to extend properties known for reversible Markov chains to Markov chains with bounded BT--dimension.

\section*{Acknowledgement}
We would like to thank Oded Goldreich for several suggestions.

\bibliographystyle{plain}
\bibliography{collective}

\end{document}